\documentclass{amsart}

\usepackage[letterpaper,top=2cm,bottom=2cm,left=3cm,right=3cm,marginparwidth=1.75cm]{geometry}

\usepackage{comment}
\usepackage{amsmath}
\usepackage{amsfonts}
\usepackage{graphicx}
\usepackage[sorting=nyt, style=alphabetic, backend=biber, sortcites]{biblatex}
\usepackage[colorlinks=true, allcolors=blue]{hyperref}

\usepackage{preamble}

\title{Product formulas for conjugated motivic Milnor basis}
\author{Hana Jia Kong}
\address{School of Mathematical Sciences, Zhejiang University, Hangzhou, China}
\email{hana.jia.kong@gmail.com}
\author{Weinan Lin}
\address{Shanghai Center for Mathematical Sciences, Fudan University, Shanghai, China}
\email{linweinan@fudan.edu.cn}

\thanks{The authors would like to thank Bert Guillou, Dan Isaksen, J.D. Quigley, Joshua Peterson, Guozhen Wang, Zhouli Xu and an anonymous referee for their insightful discussions, feedback, and comments that improved this work. The first author was supported
by “the Fundamental Research Funds for the Central Universities” No. 226-2025-00092 (China). The second author is partially supported by grants NSFC-12501085 and NSFC-12526206.
}



\begin{document}
\maketitle

\begin{abstract}
We give explicit product formulas for the mod 2 motivic Steenrod algebra in the conjugated motivic Milnor basis. This provides an efficient tool for computing the $E_2$-page of the motivic Adams spectral sequence.
\end{abstract}

\newcommand{\Car}{\mathrm{Car}}
\section{Introduction}

In recent years, motivic methods have seen significant progress and applications in stable homotopy theory. The Steenrod algebra and the Adams spectral sequence have been extended to the motivic context \cite{Voev,DI, IWX}, which provides new tools for computing motivic stable homotopy groups. 

Despite this progress, a fully explicit product formula for the conjugated motivic Milnor basis of the motivic Steenrod algebra has not yet been achieved over a general base field. It turns out that the motivic computation of such a formula is much more complex than the classical case studied by \cite{Mil}, for several reasons.

\begin{enumerate}
\item The motivic cohomology of a point can be complicated, and contains extra gradings.
\item The motivic dual Steenrod algebra is in general a Hopf algebroid; therefore, there is an asymmetry in the action of the unit map.
\item The motivic dual Steenrod algebra is not a free polynomial algebra.
\end{enumerate}

Previous work by Kylling \cite{Kyl} gives some recursive formulas in specific cases. In this paper, we work out general product formulas for the conjugated motivic Milnor basis of the mod 2 motivic Steenrod algebra. This formula also provides the foundation for machine-assisted computations of the $\mathbb R$-motivic Adams spectral sequence, {extending the work of Lin \cite{LinGrobner}}.

\section{Background}
We work in the motivic homotopy category over a field $k$, with $\mathbb F_2$ coefficients throughout.\footnote{As mentioned in \cite{Kyl}, the results also work for an essentially smooth scheme over a field, or more generally a base scheme with no points of residue
characteristic two.}

We recall some results of the motivic dual Steenrod algebra; these results can be found in \cite{Voev, Voe11, Riou, HKO, Voe03}.

Let $\mathbb M_2:= H_{**}(\Spec k;\mathbb Z/2)$ denote the motivic mod $2$ cohomology of a point. 
The Milnor conjecture \cite{Voe03} shows that $$(\mathbb M_2)_{-n,-n+*}\simeq K^M_n(k)/2[\tau], ~~|\tau|=(0,-1).$$

Here $K^M_*(k)$ denote the Milnor $K$-theory, which by definition has $K^M_1(k)=k^\times/(k^\times)^2.$ 
In particular, there is an element called $\rho:=[-1]\in K^M_1(k)\subset (\mathbb M_2)_{-1,-1}$; the element $\rho$ is possibly zero if $-1$ is a square in $k$.

The motivic dual Steenrod algebra is 
\begin{equation}\label{eq:dualA}
    A_{**}=\mathbb M_{2}[\tau_i, \xi_j, i\geq 0, j\geq 1]/(\tau_i^2=\tau\xi_{i+1} + \rho \tau_0\xi_{i+1} + \rho\tau_{i+1}).
\end{equation}
with gradings given by
$$|\tau|=(0,-1), ~~|\rho|=(-1,-1),~~|\tau_i|=(2^{i+1}-1,2^i-1), \text{ and } |\xi_i|=(2^{i+1}-2,2^i-1).$$
We adopt the notation that $\xi_0=1$.

\vspace{3em}

\newcommand{\btau}{\bar\tau}
\newcommand{\bxi}{\bar\xi}

The dual motivic Steenrod algebra is a Hopf algebroid (see, e.g.,  \cite[A.1.1.1]{ravenel} for the definition of a Hopf algebroid). In particular, it has an antipode map $\chi.$

Let $\bar\tau_i=\chi \tau_i$, $\bxi_j=\chi \xi_j$ be the conjugates of $\tau_i$, $\xi_j$ respectively. Note that $\eta_R(\tau)=\tau+\rho\tau_0$. Then (\ref{eq:dualA}) can be rewritten as

$$A_{**}=\mathbb M_{2}[\btau_i, \bxi_j, i\geq 0, j\geq 1]/(\btau_i^2=\tau\bxi_{i+1}+\rho\btau_{i+1}).$$
Here each relation has one fewer term, which makes reductions of powers of $\btau_i$ possible.

The unit map and the coproduct map in $A_{**}$, in terms of $\btau_i$ and $\bxi_j$, have the following formulas. 
\begin{enumerate}
    \item The left and right unit maps: $\eta_L(\tau)=\tau$, $\eta_R(\tau)=\tau+\rho\btau_0$, $\eta_L(\rho)=\eta_R(\rho)=\rho. $
    \item The coproduct map $A_{**}\to A_{**}\otimes_{\eta_R,\mathbb M_2,\eta_L} A_{**}$:
    $$\begin{aligned}
    \psi(\btau_k)&=1\otimes \btau_k+\sum_{i=0}^k \btau_i\otimes \bxi_{k-i}^{2^i}, \\
    \psi(\bxi_k)&=\sum_{i=0}^k \bxi_i\otimes \bxi_{k-i}^{2^i}. 
    \end{aligned}$$
\end{enumerate}

We recall that this coproduct is defined so that when we consider the motivic Steenrod algebra $A^{**}\cong \Hom_{\mathbb M_2}(A_{**}, \mathbb M_2)$ where $\mathbb M_2$ acts on $A_{**}$ via $\eta_L$, the pairing satisfies
$$\langle \alpha\otimes\beta, a\otimes b \rangle=\langle \alpha, a\langle \beta, b \rangle \rangle$$
for any $\alpha,\beta\in A_{**}$, $a,b\in A^{*,*}$.

\newcommand*{\Seq}{\mathrm{Seq}}
\newcommand*{\SeqE}{\mathscr{E}}
\newcommand*{\SeqR}{\mathscr{R}}

\begin{notation}
\label{not:milnor_basis}
We introduce the following notation about sequences.
\begin{enumerate}
\item
Let $\Seq$ denote the set of all sequences of nonnegative integers of the form $(a_0, a_1,\dots)$ such that $a_i$ is non-zero for only finitely many $i$. 
\item Let $\SeqE\subset \Seq$ denote the subset of sequences $(e_0,e_1,\dots)$ where $e_i\in\{0,1\}$. 
\item Let $\SeqR\subset \Seq$ denote the subset of sequences $(0,r_1,r_2,\dots)$ where the first entry is zero.
\item Let $(n)_i\in \Seq$ denote the sequence whose $i$th entry is $n$ and all entries are $0$.
\item It is clear that
$$\{\btau(E)\bxi(R)=\btau_0^{e_0}\btau_1^{e_1}\cdots\cdot \bxi_1^{r_1}\bxi_2^{r_2}\cdots~|~ E=(e_i)\in \SeqE, R=(r_j)\in \SeqR\}$$
is an $\bM_2$-basis of $A_{**}$. Let $\bar Q(E)$, $\bar P(R)$ and $\bar M(E,R)$ denote the duals of $\btau(E)$, $\bxi(R)$ and $\btau(E)\bxi(R)$ respectively.

\end{enumerate}
\end{notation}

\begin{remark}\label{rem:basis_change}
The relationship between the conjugated Milnor basis elements $\bar Q(E), \bar P(R)$ and the standard Milnor basis elements $Q(E), P(R)$ are determined by the antipode map $\chi$ on $A_{**}$. For example, we can obtain
$$P(0,2,0,\dots) = \bar P(0,2,0,\dots) + \tau \bar Q(1,1,0,\dots)$$
by checking that both sides evaluate to the same value on $\bxi_1^2=\xi_1^2$, $\btau_0\btau_1=\tau_0\tau_1+\tau\xi_1^2+\rho\tau_1\xi_1+\rho\tau_0\xi_1^2$, $\btau_0\bxi_1=\tau_0\xi_1$ and $\btau_1=\tau_1+\tau_0\xi_1$.
\end{remark}

We refer to the basis $\bar M(E,R)\in A^{*,*}$, $E\in \SeqE, R\in \SeqR$ as the \emph{conjugated motivic Milnor basis}. The goal of this article is to determine the product formula for
$$\bar M(E_1,R_1)\cdot \tau^n \bar M(E_2,R_2)$$
where $E_1,E_2\in \SeqE$ and $R_1,R_2\in \SeqR$.
One of the difficulties for finding the product formula comes from the relations $\btau_i^2=\tau\bxi_{i+1} + \rho\btau_{i+1}$, which will be resolved in the next section.

\section{Reduction of $\btau(S)$}
For some $S\in \Seq$ where $s_i$ could be bigger than $1$, we want to reduce $\btau(S)=\btau_0^{s_0}\btau_1^{s_1}\cdots$ to a linear combination $\Sigma c_i\btau(E_i)\bxi(R_i)$ where each $c_i\in \mathbb M_2$, $E_i\in \SeqE$ and $R_i\in \SeqR$.
\begin{example}\label{ex:tau-red}
We have
\begin{align*}
\btau_0^2\btau_1 = \tau \btau_1\bxi_1 + \rho \btau_1^2
=\tau \btau_1\bxi_1 + \rho\tau\bxi_2 + \rho^2\btau_2,
\end{align*}
and 
\begin{align*}
\btau_0^4& = \tau^2\bxi_1^2
= \tau^2\bxi_1^2 + \rho^2\btau_1^2
= \tau^2\bxi_1^2 + \rho^2\tau\bxi_2 + \rho^3\btau_2.
\end{align*}
\end{example}

\begin{construction}
We construct a binary tree to encode the reduction of $\btau(S)$ for $S\in \Seq$ to a linear combination of $\btau(E)\bxi(R)$ for $E\in\SeqE$, $R\in\SeqR$.
\begin{enumerate}
\item 
Every node is labeled by two sequences $S'|R'$, representing $\btau(S')\bxi(R')$. Note that different nodes can share the same label.
\item The sum of the monomials represented by child nodes equals the monomial represented by the parent node.
\item The branches from a parent node to its child nodes correspond to rewriting a $\btau_i^2$ factor (where $i$ is as small as possible) into $\tau\bxi_{i+1}+\rho\btau_{i+1}$.
\end{enumerate}

We start from $\btau(S)$.
\begin{enumerate}
\item
The root node is labeled by $S|0$. 
\item 
A node labeled $S'|R'$ with $S'\in \Seq\backslash \SeqE $ has a left child node labeled
$$S'-(2)_i+(1)_{i+1}|R',$$
and a right child node labeled
$$S'-(2)_i| R'+ (1)_{i+1},$$
where $i$ is the smallest index with $s_i\geq 2$.
\item A node labeled $E|R$ with $E\in\SeqE$, $R\in\SeqR$ is a leaf node.
\end{enumerate}

Once we construct the tree, we can reduce $\btau(S)$ to
$$\btau(S) = \sum c\cdot \tau^k\rho^l\btau(E)\bxi(R),$$
where $c$ is the number of occurrences of the leaf nodes labeled $E|R$, and the powers of $\tau$ and $\rho$ can be determined by the degrees.
\end{construction}

For example, the reduction of $\btau_0^2\btau_1$ labeled $(21)|(0)$ and $\btau_0^4$ labeled $(4)|(0)$ from Example \ref{ex:tau-red} is encoded by the following trees.
$$
\begin{forest}
  for tree={
    draw,
    minimum size=1em, 
    s sep=1em, 
    l sep=0.7em, 
  }
  [$(21)|(0)$
    [$(02)|(0)$
      [$(001)|(0)$]
      [$(0)|(001)$]
    ]
    [$(01)|(01)$
    ]
  ]
\end{forest}
\begin{forest}
  for tree={
    draw,
    minimum size=1em, 
    s sep=1em, 
    l sep=0.7em, 
  }
  [$(4)|(0)$
    [$(21)|(0)$
      [$(02)|(0)$
        [$(001)|(0)$]
        [$(0)|(001)$]
        ]
      [$(01)|(01)$]
    ]
    [$(2)|(01)$
        [$(01)|(01)$]
        [$(0)|(02)$]
    ]
  ]
\end{forest}
$$

\begin{notation}
For $S\in Seq$, $R\in \SeqR$ we introduce the following notations.
\begin{enumerate}
\item
$ \wsum S=\sum s_i2^i.$
\item
$ \nsum S=\sum s_i.$
\item $c(S,R)= 
\prod_{n\geq 1} \binom{\lfloor\sum_{i=0}^{n-1} 2^{i-n}(s_i - r_i)\rfloor}{r_n}.$
Here we take  $\binom{m}{n}=0$ if $m<n$.
\end{enumerate}
\end{notation}

\begin{theorem}
\label{thm:simp}
\begin{equation}\label{eq:simp}
\btau(S)=\sum_{\substack{E\in \SeqE, R\in \SeqR\\ \wsum E + \wsum R=\wsum S}} c(S, R)\tau^{\nsum R}\rho^{\nsum S -\nsum E-2\nsum R}\btau(E){\bxi(R)}
\end{equation}

\end{theorem}

\begin{proof}
Note that 
$$\wsum E + \wsum R =\wsum S$$
because of the rewriting rule: every two in $s_i$ contributes to one in $e_{i+1}$ or one in $r_{i+1}$. The sequence $E$ is actually determined uniquely by the equation. The power of $\tau$ and $\rho$ can be verified by comparing degrees and it remains to show that the number of occurrences of leaf nodes labeled $E|R$ in the rewriting tree of $\btau(S)$ is $c(S,R)$.

We use a bottom-up approach to count all leaf nodes labeled $E|R$ with given $E=(e_0,e_1,\dots)\in \SeqE$ and $R=(0,r_1,r_2,\dots)\in \SeqR$. Let $n$ be the largest index such that $r_n\neq 0$. Given any $m<n$, since we always rewrite $\btau_i^2$ from the left-most re-writable entry, before any $\btau_m^2$ is rewritten from $\btau(S)$ its descendants labeled $S'|R'$ always satisfy
$$\sum_{k=0}^{m-1}2^k {s'_k}+\sum_{k=1}^m 2^k{r'_k}=\sum_{k=0}^{m-1}2^k {s_k}.$$
Hence a node labeled $S_n|R_n=E|R$ is a descendant of a node labeled
$$S_{n-1}|R_{n-1}= (e_0,\dots, e_{n-2},\sum_{i=0}^{n-2} 2^{i-n+1}(s_i - e_i - r_i)+s_{n-1}-r_{n-1},s_n,\dots)|(0,r_1,\dots, r_{n-1}, 0,\dots)$$
and we have
$$\sum_{i=0}^{n-2} 2^{i-n+1}(s_i - e_i - r_i)+s_{n-1}-r_{n-1} = \left\lfloor\sum_{i=0}^{n-1} 2^{i-n+1}(s_i - r_i)\right\rfloor$$
because $\sum_{i=0}^{n-2} 2^{i-n+1}e_i<1$.

On the other hand, a node labeled $S_{n-1}|R_{n-1}$ has exactly 
$$\binom{\lfloor\lfloor\sum_{i=0}^{n-1} 2^{i-n+1}(s_i - r_i)\rfloor / 2\rfloor}{r_{n}} = \binom{\lfloor \sum_{i=0}^{n-1} 2^{i-n}(s_i - r_i)\rfloor}{r_{n}}$$ 
descendants labeled $E_n|R_n$ because

$$\btau_{n-1}^{\left\lfloor\sum_{i=0}^{n-1} 2^{i-n+1}(s_i - r_i)\right\rfloor}=\btau_{n-1}^{e_{n-1}}(\tau\bxi_n+\rho\btau_n)^{\lfloor \sum_{i=0}^{n-1} 2^{i-n}(s_i - r_i)\rfloor}.$$

Next, we count the occurrences of $S_{n-1}|R_{n-1}$ in descendants of a node labeled 
$$S_{n-2}|R_{n-2}= (e_0,\dots, e_{n-3},\sum_{i=0}^{n-3} 2^{i-n+2}(s_i - e_i - r_i)+s_{n-2}-r_{n-2},s_{n-1},\dots)|(0,r_1,\dots, r_{n-2}, 0,\dots)$$
and repeat the same calculation, until we get to the root node
$$S_0|R_0 = S|(0).$$
We conclude that the number of occurrences of leaf nodes labeled $E|R$ is 
$$c(S,R)=
\prod_{m=1}^n \binom{\lfloor\sum_{i=0}^{m-1} 2^{i-m}(s_i - r_i)\rfloor}{r_m}$$
\end{proof}

\begin{corollary}
    If $c(S,R)\neq 0$, then $2\nsum R\le \nsum S - \nsum E$.
\end{corollary}

\section{Coproduct formula for $\btau(E)\bxi(T)$}
\begin{notation}
Let $\mathscr X$ be the set of matrices of the following form.
$$\begin{pmatrix}
x_{0,0} & x_{0,1} & x_{0,2} & \cdots\\
x_{1,0} & x_{1,1} & x_{1,2} & \cdots\\
x_{2,0} & x_{2,1} & x_{2,2} & \cdots\\
\vdots & \vdots& \vdots&
\end{pmatrix}$$

\begin{enumerate}
\item We define $\mathscr X_0\subset \mathscr X$ as the subset of $X\in \mathscr X$ such that $x_{0,0}=0$.
\item For $X\in \mathscr X$, define $T,S,R: \mathscr X\to \Seq$ by
$$T(X)_k=\sum_{i} x_{i,k-i},~~
S(X)_k=\sum_j x_{k,j},~~
R(X)_k=\sum_i 2^ix_{i,k}.$$
We define $T_0,S_0,R_0:\mathscr X\to \SeqR$ by setting $F_0(X)=(0,F(X)_1,F(X)_2,\dots)$ for $F=T,S,R$.
\item For $X\in \mathscr X$, define
$$b(X) = \prod_{i\ge 1} \frac{(x_{i,0}+x_{i-1,1}+\cdots+x_{0,i})!}{x_{i,0}!x_{i-1,1}!\cdots x_{0,i}!} ( \text{mod } 2).$$
\item Assume $R,R'\in \Seq$. We write $R\leq R'$ if $r_i\leq r_i'$ for all $i$.
\end{enumerate}
\end{notation}

\vspace{2em}

\begin{lemma} 
\label{lem:coprod}
For any $E\in\SeqE, T\in \Seq$, we have
\begin{enumerate}
\item
\begin{align*}
    \psi(\btau(E)) =& \sum_{\substack{Y\in \mathscr X\\T(Y)\le E}} b(Y)\cdot \btau(S(Y))\otimes \btau(E-T(Y)) \bxi(R_0(Y)).
\end{align*}
\item
\begin{align*}
\psi(\bxi(T)) = \sum_{\substack{X\in \mathscr X_0\\T_0(X)=T}} b(X) \cdot\bxi(S_0(X))\otimes \bxi(R_0(X)).
\end{align*}
\end{enumerate}
\end{lemma}
\begin{proof}
The proof of (2) is the same as the classical case in \cite{Mil}. 

The equation in (1) can be computed by the following.
\begin{align*}
\psi(\btau(E)) =& \prod_k (1\otimes \btau_k+\sum_{i=0}^k \btau_i\otimes \bxi_{k-i}^{2^i})^{e_k}\\
=& \sum_{E'+E'' = E} \prod_k (1\otimes \btau_k)^{e_k'} \cdot (\sum_{i=0}^k \btau_i\otimes \bxi_{k-i}^{2^i})^{e^{\prime\prime}_k}\\
=& \sum_{\substack{E'+E'' = E\\ Y\in \mathscr X,T(Y)=E''}} b(Y) \cdot \btau(S(Y))\otimes \btau(E') \bxi(R_0(Y))\\
=&\sum_{\substack{ Y\in \mathscr X,T(Y) \leq E}} b(Y) \cdot \btau(S(Y))\otimes \btau(E-T(Y)) \bxi(R_0(Y)).
\end{align*}

\end{proof}

\begin{corollary}\label{cor:coprod_simp}
For any $E\in \SeqE, T\in \SeqR,$ we have
\begin{align*}
    & \psi(\btau(E)\bxi(T))
    \\
    =& \sum_{\substack{Y\in \mathscr X,T(Y)\le E\\X\in \mathscr X_0, T_0(X)=T}}
    b(Y)b(X)\cdot \btau(S(Y))\bxi(S_0(X))\otimes \btau(E-T(Y)) \bxi(R_0(Y+X))\\
    =& \sum_{\substack{Y\in \mathscr X,T(Y)\le E\\X\in \mathscr X_0, T_0(X)=T}}~
    \sum_{\substack{E'\in \SeqE, R'\in \SeqR\\ \wsum{E'} + \wsum{R'} =\wsum{S(Y)}}}
    \begin{array}{|c|}\hline
    b(Y)b(X)c(S(Y), R')\\
    \cdot\\
    \tau^{\nsum{R'}}\rho^{\nsum{S(Y)} -\nsum{E'}-2\nsum{R'}}\\
    \cdot\\
    \btau(E')\bxi(R'+S_0(X))\otimes \btau(E-T(Y)) \bxi(R_0(Y+X))\\\hline
    \end{array}.
\end{align*}
(We write the long product in a multi-line box to improve readability.)
\end{corollary}
\begin{proof}
It follows from \cref{lem:coprod} and \cref{thm:simp}.
\end{proof}

\section{Product formula for the conjugated motivic Milnor basis}
\begin{theorem}\label{thm:MM_prod}
For any $E_1,E_2\in \SeqE$ and $R_1,R_2\in \SeqR$, we have
\begin{align*}
    & \bar M(E_1,R_1) \cdot \bar M(E_2,R_2) \\
    =& \sum_{\substack{Y\in \mathscr X,X\in \mathscr X_0\\ R_0(Y+X)=R_2\\S_0(X)\le R_1\\\wsum{E_1} + \wsum{R_1-S_0(X)} =\wsum{S(Y)}\\
    E_2+T(Y)\in \SeqE}}
    \begin{array}{|c|}\hline
    b(Y)b(X)c(S(Y), R_1-S_0(X))\\
    \cdot\\
    \tau^{\nsum{R_1-S_0(X)}}\rho^{\nsum{S(Y)} -\nsum{E_1}-2\nsum{R_1-S_0(X)}}\\
    \cdot\\
    \bar M(E_2+T(Y),T_0(X))\\\hline
    \end{array}.
\end{align*}
\end{theorem}
\begin{proof}
    This is the dualization of \cref{cor:coprod_simp} with substitutions $E'=E_1$, $R'+S_0(X)=R_1$, $E-T(Y)=E_2$ and $R_0(Y+X)=R_2$.
\end{proof}

Lastly, we deal with the case when the second factor has a coefficient $\tau^n$.
\begin{theorem}
For $E\in \SeqE, R\in\SeqR,$ we have
\label{thm:QPtau_prod}
$$\bar M(E,R)\cdot \tau^n 
   = \sum_{\substack{0\le m\le n\\E'\in \SeqE, R'\in \SeqR\\ \wsum{E} + \wsum{R'}=\wsum{E'}+m}} 
    \begin{array}{|c|}\hline
    \binom{n}{m}c(E'+(m)_0,R')\\
    \cdot\\
    \tau^{\nsum{R'}+n-m}\rho^{\nsum{E'} -\nsum E-2\nsum {R'}+2m}\\
    \cdot\\
    \bar M(E',R-R')\\\hline
    \end{array}.$$
\end{theorem}

\begin{proof}
Note that
for $b\in A_{**}$ and $a\in A^{*,*}$, we have
$$\langle b, a\tau^n\rangle=\langle b, a\langle \tau^n, 1\rangle\rangle=\langle b\otimes \tau^n, a\otimes 1\rangle=\langle(\tau+\rho\btau_0)^nb, a\rangle.$$
For $E'\in\SeqE$ and $T\in\SeqR$, we have
\begin{align*}
    & (\tau+\rho\btau_0)^n\btau(E')\bxi(T)\\
    =& \sum_m\binom{n}{m} \tau^{n-m}\rho^m\tau\left(E'+(m)_0\right)\bxi(T)\\
    =& \sum_m\sum_{\substack{E\in \SeqE, R'\in \Seq\\ \wsum{E} + \wsum{R'}=\wsum{E'}+m}} 
    \begin{array}{|c|}\hline
    \binom{n}{m}c(E'+(m)_0,R')\\
    \cdot\\
    \tau^{\nsum{R'}+n-m}\rho^{\nsum{E'} -\nsum {E}-2\nsum {R'}+2m}\\
    \cdot\\
    \btau(E)\bxi(R'+T) \\\hline
    \end{array}.
\end{align*}
The conclusion is dual to the above with substitution $R=R'+T$.
\end{proof}

\begin{corollary}\label{cor:last}
    For $E_1, E_2\in \SeqE, R_1,R_2\in\Seq,$ we have
\begin{align*}
    & \bar M(E_1,R_1)\cdot \tau^n \bar M(E_2,R_2)\\
    =& \sum_{\substack{0\le m\le n\\E'\in \SeqE, R'\in \SeqR\\ \wsum{E_1} + \wsum{R'}=\wsum{E'}+m}} 
    \begin{array}{|c|}\hline
    \binom{n}{m}c(E'+(m)_0,R')\\
    \cdot\\
    \tau^{\nsum{R'}+n-m}\rho^{\nsum{E'} -\nsum {E_1}-2\nsum {R'}+2m}\\
    \cdot\\
    \bar M(E',R_1-R')\cdot \bar M(E_2,R_2)\\\hline
    \end{array}\\
    =& \sum_{\substack{0\le m\le n\\E'\in \SeqE, R'\in \SeqR\\ \wsum{E_1} + \wsum{R'}=\wsum{E'}+m\\
    Y\in \mathscr X,X\in \mathscr X_0\\R_0(Y+X)=R_2\\S_0(X)\le R_1-R'\\\wsum{E'} + \wsum{R_1-R'-S_0(X)} =\wsum{S(Y)}\\E_2+T(Y)\in \SeqE}} 
    \begin{array}{|c|}\hline
    \binom{n}{m}c(E'+(m)_0,R')\cdot b(Y)b(X)c(S(Y), R_1-R'-S_0(X))\\
    \cdot\\
    \tau^{\nsum{R_1-S_0(X)}+n-m}\rho^{\nsum{S(Y)} -\nsum{E_1}+2\nsum{S_0(X)}-2\nsum{R_1}+2m}\\
    \cdot\\
    \bar M(E_2+T(Y),T_0(X))\\\hline
    \end{array}.\\
\end{align*}
\end{corollary}

The following corollary is a conjugated version of \cite[Section 13]{Voev}.
\begin{corollary}
Consider $E\in \SeqE$ and $R\in \SeqR$. We have
\begin{enumerate}
    \item $\bar M(E, R) = \bar P(R)\bar Q(E)$,
    \item $\bar Q(e_0,e_1,\cdots) = \prod_{e_i\neq 0} \bar Q_i, \text{ where } e_i\in \{0,1\},$
    \item $\bar Q_i\bar Q_j=\bar Q_j\bar Q_i$,
    \item $\bar Q_i^2 = 0$.
\end{enumerate}
\end{corollary}
\begin{proof}
(1) By Theorem \ref{thm:MM_prod}, we have
\begin{align*}
    & \bar P(R)\bar Q(E) = \bar M(0,R) \cdot \bar M(E,0) \\
    =& \sum_{\substack{Y\in \mathscr X,X\in \mathscr X_0\\ R_0(Y+X)=0\\S_0(X)\le R\\\wsum{R-S_0(X)} =\wsum{S(Y)}\\
    E+T(Y)\in \SeqE}}
    \begin{array}{|c|}\hline
    b(Y)b(X)c(S(Y), R-S_0(X))\\
    \cdot\\
    \tau^{\nsum{R-S_0(X)}}\rho^{\nsum{S(Y)}-2\nsum{R-S_0(X)}}\\
    \cdot\\
    \bar M(E_2+T(Y),T_0(X))\\\hline
    \end{array}\\
    =& \sum_{\substack{\wsum{R'} =\wsum{E'}\\
    E+E'\in \SeqE}}
    c(E', R')\cdot \tau^{\nsum{R'}}\rho^{\nsum{E'}-2\nsum{R'}}\cdot \bar M(E+E',R-R')\\
    =& \bar M(E,R)
\end{align*}
The third equality holds because $R_0(Y+X)=0$ implies that $Y$ and $X$ are nonzero only in the zeroth column. The fourth equality holds because $\wsum{R'}=\wsum{E'}$ and $E'\in \SeqE$ implies $\nsum{R'}\ge \nsum{E'}$. On the other hand nonzero terms always have $2\nsum{R'}\le \nsum{E'}$ and then we must have $R'=E'=0$. Hence there is only one nontrivial term.

(2) It is sufficient to prove when $E+(1)_j\in \SeqE$ we have $\bar M(E+(1)_j, 0)=\bar M(E, 0)\bar Q_j$.
By Theorem \ref{thm:MM_prod}, we have
\begin{align*}
    & \bar M(E, 0)\bar Q_j = \bar M(E, 0)\cdot \bar M((1)_j, 0) \\
    =& \sum_{\substack{Y\in \mathscr X\\ R_0(Y)=0\\\wsum{E} =\wsum{S(Y)}\\
    (1)_j+T(Y)\in \SeqE}}
    b(Y)\cdot \rho^{\nsum{S(Y)} -\nsum{E}}\cdot \bar M((1)_j+T(Y),0)\\
    =& \sum_{\substack{\wsum{E} =\wsum{S}\\
    (1)_j+S\in \SeqE}}
    \rho^{\nsum{S} -\nsum{E}}\cdot \bar M((1)_j+S,0)\\
    =& \bar M(E+(1)_j, 0)
\end{align*}
The last equality holds because $(1)_j+S\in \SeqE$ implies $S\in \SeqE$. Hence $\wsum E=\wsum S$ forces $S=E$.

(3) This is a direct consequence of (2).

(4) By Theorem \ref{thm:MM_prod}, we have
\begin{align*}
    & \bar Q_i^2 = \bar M((1)_i, 0)\cdot \bar M((1)_i, 0) \\
    =& \sum_{\substack{Y\in \mathscr X\\ R_0(Y)=0\\\wsum{(1)_i} =\wsum{S(Y)}\\
    (1)_i+T(Y)\in \SeqE}}
    b(Y)\cdot \rho^{\nsum{S(Y)} -\nsum{(1)_i}}\cdot \bar M((1)_i+T(Y),0)\\
    =& \sum_{\substack{\wsum{(1)_i} =\wsum{S}\\
    (1)_i+S\in \SeqE}}
    \rho^{\nsum{S} -\nsum{(1)_i}}\cdot \bar M((1)_i+S,0)\\
    =& 0
\end{align*}
The last equality holds because $\wsum{(1)_i} =\wsum{S}$ and $(1)_i+S\in \SeqE$ cannot hold simultaneously.
\end{proof}

\begin{example}\label{ex:Kylling}
We use our formula to compute an example and compare it to \cite{Kyl}. Using the standard basis, Kylling calculated:
\begin{equation*}
    P(0,2,0,\dots) \cdot Q(0,1,0,\dots) = Q(0,1,0,\dots) P(0,2,0,\dots) + Q(0,0,1,0,\dots) + \rho Q(1,1,0,\dots) P(0,1,0,\dots).
\end{equation*}
(Note that $P(0,2,0,\dots)$ is denoted by $P^2$ and $Q(0,1,0,\dots)$ by $Q_1$ in \cite{Kyl}.)

We can recover this result using our formula. As mentioned in Remark \ref{rem:basis_change}, $$P(0,2,0,\dots) = \bar P(0,2,0,\dots)+\tau \bar Q(1,1,0,\dots)$$ and $$Q(0,1,0,\dots) = \bar Q(0,1,0,\dots).$$ Therefore, by \cref{cor:last},
\begin{align*}
P(0,2,0,\dots) \cdot Q(0,1,0,\dots) &= \left(\bar P(0,2,0,\dots)+\tau \bar Q(1,1,0,\dots)\right) \cdot \bar Q(0,1,0,\dots)\\
&= \bar P(0,2,0,\dots) \bar Q(0,1,0,\dots) + \tau \bar Q_0 \bar Q_1^2 \\
&= \bar P(0,2,0,\dots) \bar Q(0,1,0,\dots)\\
&= \bar M((0,1,0,\dots), (0,2,0,\dots)).
\end{align*}

It remains to show that $\bar M((0,1,0,\dots), (0,2,0,\dots))$ agrees with the right-hand side of the above equation by Kylling.
This can be checked by the following conversion between classical basis and conjugated basis of $A_{*,*}$ in relative degrees.
\begin{align*}
    \chi(\tau_1 \xi_1^2) &= (\tau_1 + \tau_0 \xi_1)\xi_1^2 = \tau_1 \xi_1^2 + \tau_0 \xi_1^3, \\
    \chi(\tau_2) &= \tau_2 + \tau_1 \xi_1^2 + \tau_0 \xi_2 + \tau_0 \xi_1^3, \\
    \chi(\tau_0 \tau_1 \xi_1) &= \tau_0 (\tau_1 + \tau_0 \xi_1) \xi_1 = \tau_0 \tau_1 \xi_1 + \tau_0^2 \xi_1^2 \\
    &= \tau_0 \tau_1 \xi_1 + (\tau \xi_1 + \rho \tau_1 + \rho \tau_0 \xi_1)\xi_1^2 \\
    &= \tau_0 \tau_1 \xi_1 + \tau \xi_1^3 + \rho \tau_1 \xi_1^2 + \rho \tau_0 \xi_1^3.
\end{align*}
\end{example}

\printbibliography

@article{Kyl,
  title={Recursive formulas for the motivic Milnor basis},
  author={Kylling, Jonas Irgens},
  journal={New York journal of mathematics},
  volume={23},
  pages={49--58},
  year={2017}
}

@article{Voe03,
  title={Motivic cohomology with $\mathbb Z/2$-coefficients},
  author={Voevodsky, Vladimir},
  journal={Publications Math{\'e}matiques de l'IH{\'E}S},
  volume={98},
  pages={59--104},
  year={2003}
}

@article{Riou,
  title={Op\'erations de Steenrod motiviques},
  author={Riou, Jo{\"e}l},
  journal={arXiv preprint arXiv:1207.3121},
  year={2012}
}

@article{HKO,
  title={The motivic Steenrod algebra in positive characteristic.},
  author={Hoyois, Marc and Kelly, Shane and {\O}stv{\ae}r, Paul Arne},
  journal={Journal of the European Mathematical Society (EMS Publishing)},
  volume={19},
  number={12},
  year={2017}
}

@article{Voe11,
  title={On motivic cohomology with $\mathbb Z/l$-coefficients},
  author={Voevodsky, Vladimir},
  journal={Annals of mathematics},
  pages={401--438},
  year={2011},
  publisher={JSTOR}
}

@article{IWX,
  title={Stable homotopy groups of spheres: from dimension 0 to 90},
  author={Isaksen, Daniel C and Wang, Guozhen and Xu, Zhouli},
  journal={Publications math{\'e}matiques de l'IH{\'E}S},
  volume={137},
  number={1},
  pages={107--243},
  year={2023},
  publisher={Springer}
}

@article{DI,
  title={The motivic Adams spectral sequence},
  author={Dugger, Daniel and Isaksen, Daniel C},
  journal={Geometry \& Topology},
  volume={14},
  number={2},
  pages={967--1014},
  year={2010},
  publisher={Mathematical Sciences Publishers}
}

@article{Mil,
  title={The Steenrod algebra and its dual},
  author={Milnor, John},
  journal={Annals of Mathematics},
  volume={67},
  number={1},
  pages={150--171},
  year={1958},
  publisher={JSTOR}
}

@article{Voev,
  title={Reduced power operations in motivic cohomology},
  author={Voevodsky, Vladimir},
  journal={Publications Math{\'e}matiques de l'IH{\'E}S},
  volume={98},
  pages={1--57},
  year={2003}
}

@misc{LinGrobner,
    AUTHOR = {Weinan Lin.},
     TITLE = {Noncommutative {G}r\"obner Bases and {E}xt Groups. {P}eking {M}athematical {J}ournal, 2023.},
howpublished   = {\url{https://doi.org/10.1007/s42543-023-00080-6} Published Online}
}

@book{ravenel,
Author = {Ravenel, D.},
Title = {
 Complex Cobordism and Stable Homotopy Groups of Spheres},
 journal = {American Mathematical Society (AMS): Providence, RI, United States, 2003}
}

\end{document}